\documentclass{article}

\usepackage[a4paper, top=3cm, left=2.5cm, right=2.5cm, bottom=3cm]{geometry}
\usepackage[dvips]{graphicx}
\usepackage[active]{srcltx}
\usepackage{amsmath, amsthm, pb-diagram}
\usepackage{amssymb}
\usepackage{amsfonts}
\usepackage{enumerate, fancyhdr, dsfont}

    \newcommand{\ran}{\mbox{\rm ran}}

    \newcommand{\lcom}{\textquotedblleft}
    \newcommand{\rcom}{\textquotedblright}

    \newcommand{\modulo}{\mbox{\rm\ mod\ }}

    \newcommand{\cj}[2]{ \left\{ {#1} \ / \ {#2} \right\} }

    \newcommand{\thzfc}{\mathrm{ZFC}}

    \newcommand{\Cwf}{\mathcal{C}}

    \newcommand{\Fwf}{\mathcal{F}}

    \newcommand{\Iwf}{\mathcal{I}}
    
    \newcommand{\Mwf}{\mathcal{M}}
    \newcommand{\Nwf}{\mathcal{N}}

    \newcommand{\Uwf}{\mathcal{U}}
    \newcommand{\Vwf}{\mathcal{V}}

    \newcommand{\afrak}{\mathfrak{a}}
    \newcommand{\bfrak}{\mathfrak{b}}
    \newcommand{\cfrak}{\mathfrak{c}}
    \newcommand{\dfrak}{\mathfrak{d}}
    \newcommand{\ffrak}{\mathfrak{f}}
    
    \newcommand{\pfrak}{\mathfrak{p}}
    \newcommand{\rfrak}{\mathfrak{r}}
    \newcommand{\sfrak}{\mathfrak{s}}
    \newcommand{\tfrak}{\mathfrak{t}}
    \newcommand{\ufrak}{\mathfrak{u}}

    \newcommand{\menos}{\smallsetminus}

    \newcommand{\frestr}{\!\!\upharpoonright\!\!}

    \newcommand{\add}{\mbox{\rm add}}
    \newcommand{\cov}{\mbox{\rm cov}}
    \newcommand{\non}{\mbox{\rm non}}
    \newcommand{\cof}{\mbox{\rm cof}}
    \newcommand{\Sl}{\mbox{\rm Sl}}

    \newcommand{\Aor}{\mathds{A}}
    \newcommand{\Bor}{\mathds{B}}
    \newcommand{\Cor}{\mathds{C}}
    \newcommand{\Dor}{\mathds{D}}
    \newcommand{\Eor}{\mathds{E}}
    \newcommand{\Lor}{\mathds{L}}
    \newcommand{\Mor}{\mathds{M}}
    \newcommand{\Por}{\mathds{P}}
    \newcommand{\Qor}{\mathds{Q}}
    
    \newcommand{\Sor}{\mathds{S}}
    
    \newcommand{\Qnm}{\dot{\mathds{Q}}}
    \newcommand{\Rnm}{\dot{\mathds{R}}}
    \newcommand{\Snm}{\dot{\mathds{S}}}
    \newcommand{\Tnm}{\dot{\mathds{T}}}
    \newcommand{\Anm}{\dot{\mathds{A}}}
    \newcommand{\Bnm}{\dot{\mathds{B}}}
    \newcommand{\Cnm}{\dot{\mathds{C}}}
    \newcommand{\Dnm}{\dot{\mathds{D}}}
    \newcommand{\Enm}{\dot{\mathds{E}}}
    \newcommand{\Ufnm}{\dot{\mathcal{U}}}

    \newcommand{\cf}{\mbox{\rm cf}}

    \newcommand{\disy}{{\ \mbox{\scriptsize $\vee$} \ }}

    \newcommand{\sii}{{\ \mbox{$\Leftrightarrow$} \ }}

\title{Models of some cardinal invariants with large continuum}
\author{Diego Alejandro Mej\'ia\thanks{Supported by the Monbukagakusho (Ministry of Education, Culture, Sports, Science and Technology) Scholarship, Japan.}
}

\date{\small Graduate School of System Informatics, Kobe University, Kobe, Japan.\\ \texttt{damejiag@kurt.scitec.kobe-u.ac.jp}}

\begin{document}

\makeatletter
\def\@roman#1{\romannumeral #1}
\makeatother

\theoremstyle{plain}
  \newtheorem{theorem}{Theorem}[section]
  \newtheorem{corollary}[theorem]{Corollary}
  \newtheorem{lemma}[theorem]{Lemma}
  \newtheorem{prop}[theorem]{Proposition}
  \newtheorem{clm}[theorem]{Claim}
  \newtheorem{exer}[theorem]{Exercise}
  \newtheorem{question}[theorem]{Question}
\theoremstyle{definition}
  \newtheorem{definition}[theorem]{Definition}
  \newtheorem{example}[theorem]{Example}
  \newtheorem{remark}[theorem]{Remark}
  \newtheorem{context}[theorem]{Context}
  \newtheorem*{acknowledgements}{Acknowledgements}

\maketitle

\begin{abstract}
   We extend the applications of the techniques used in \cite{mejia} to present various examples of consistency results where some cardinal invariants of the continuum take arbitrary regular values with the size of the continuum being bigger than $\aleph_2$.
\end{abstract}

\section{Introduction}\label{SecIntro}

We use the fsi (finite support iteration) techniques presented in \cite{brendle} and the matrix iterations technique introduced by Blass and Shelah in \cite{blsh} and implemented in \cite{BF} and \cite{mejia} to construct models where the continuum is large (that is, its size is bigger than $\aleph_2$) and where the cardinal invariants of the continuum mentioned in this section take arbitrary regular values.\\
We introduce the notation and the cardinal invariants that concern the contents of this text. Our notation is quite standard. $\Aor$ represents the amoeba algebra, $\Bor$ the random algebra, $\Cor$ the Cohen poset, $\Dor$ is Hechler forcing, $\Eor$ is the eventually different reals forcing and $\mathds{1}$ denotes the trivial poset $\{0\}$. Those posets are Suslin ccc forcing notions. See \cite{barju} for definitions and properties. Basic notation and knowledge about forcing can be found in \cite{kunen} and \cite{jech}.\\
Throughout this text, we refer as a \emph{real} to any member of a fixed Polish space (e.g. the Baire space $\omega^\omega$ or the Cantor space $2^\omega$). $\Mwf$ denotes the $\sigma$-ideal of meager sets of reals and $\Nwf$ is the $\sigma$-ideal of null sets of reals (from the context, it is clear which Polish space corresponds to such an ideal). For $\Iwf$ being $\Mwf$ or $\Nwf$, the following cardinal invariants are defined:
\begin{description}
   \item[$\add(\Iwf)$] the least size of a family $\Fwf\subseteq\Iwf$ whose union is not in $\Iwf$,
   \item[$\cov(\Iwf)$] the least size of a family $\Fwf\subseteq\Iwf$ whose union covers all the reals,
   \item[$\non(\Iwf)$] the least size of a set of reals not in $\Iwf$, and
   \item[$\cof(\Iwf)$] the least size of a cofinal subfamily of $\langle\Iwf,\subseteq\rangle$.
\end{description}
The value of each of these invariants does not depend on the space of reals used to define it.\\
We consider $\cfrak=2^{\aleph_0}$ (the size of the continuum) and the invariants $\bfrak$ and $\dfrak$ as given in Section \ref{SecPres}.
Thus, we have Cichon's diagram as in figure \ref{fig:1}.
\begin{figure}
\begin{center}
  \includegraphics[scale=0.52]{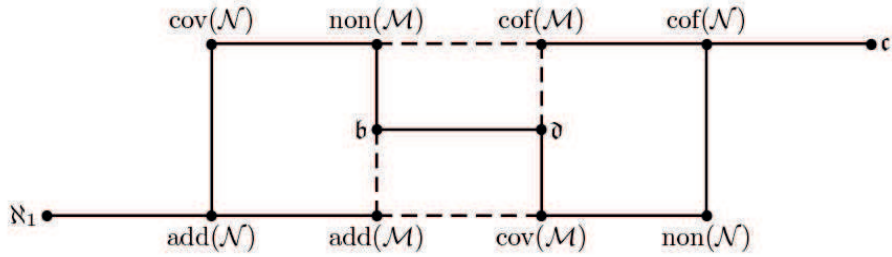}
\caption{Cichon's diagram}
\label{fig:1}
\end{center}
\end{figure}
In figure \ref{fig:1}, horizontal lines from left to right and vertical lines from down to up represent $\leq$. The dotted lines represent
$\add(\Mwf)=\min\{\bfrak,\cov(\Mwf)\}$ and $\cof(\Mwf)=\max\{\dfrak,\non(\Mwf)\}$. For basic definitions, notation and proofs regarding Cichon's
diagram, see \cite[Chapter 2]{barju}, \cite{blass} and \cite{bart}.\\
For $A$ and $B$ subsets of $\omega$, $A\subseteq^* B$ denotes that $A\menos B$ is finite. A family $\Fwf$ contained in $[\omega]^\omega$ is a \emph{filter base} if the intersection of any finite subfamily of $\Fwf$ is infinite and $\bar{\Fwf}:=\cj{X\in[\omega]^\omega}{\exists_{F\in[\Fwf]^{<\omega}}\bigcap F\subseteq^* X}$ denotes the filter that it generates.
$X\in[\omega]^\omega$ is said to be a \emph{pseudo-intersection of $\Fwf$} if $X\subseteq^* A$ for any $A\in\Fwf$. The cardinal invariant $\pfrak$, the \emph{pseudo-intersection number}, is defined as the least size of a filter base that does not have a pseudo-intersection, and the cardinal invariant $\ufrak$, the \emph{ultrafilter number}, is the least size of a filter base that generates a non-principal ultrafilter on $\omega$. The cardinal invariants $\sfrak$ and $\rfrak$ are defined in Section \ref{SecPres}. For a filter base $\Fwf$, $\Mor_{\Fwf}$ denotes \emph{Mathias forcing with $\Fwf$}, which is a $\sigma$-centered forcing notion that adds a pseudo-intersection of $\Fwf$. For definitions, properties and proofs, see \cite{barju} and \cite{blsh}.\\
It is clear that the mentioned cardinal invariants are between $\aleph_1$ and $\cfrak$. The following are the inequalities that are known to be true in $\thzfc$. Refer to \cite{blass} and \cite{barju} for the proofs.
\begin{theorem}
   \begin{enumerate}[(a)]
     \item $\pfrak\leq\add(\Mwf)$.
     \item $\pfrak\leq\sfrak$.
     \item $\sfrak\leq\dfrak$ and $\bfrak\leq\rfrak$.
     \item $\sfrak\leq\non(\Iwf)$ and $\cov(\Iwf)\leq\rfrak$, where $\Iwf$ is $\Mwf$ or $\Nwf$.
     \item $\rfrak\leq\ufrak$.
   \end{enumerate}
\end{theorem}
In fact, (a) and (b) are proved for the \emph{tower number} $\tfrak$ in place of $\pfrak$, which is the least length of a well-ordered $\subseteq^*$-decreasing sequence of infinite subsets of $\omega$ which has no pseudo-intersection. It is well known that $\pfrak\leq\tfrak$ (from which (a) and (b) follows), but the problem whether $\pfrak=\tfrak$ is provable in $\thzfc$ was a long standing question recently answered positively by Malliaris and Shelah in \cite{mallsh}.\\
This paper is structured as follows. In Section \ref{SecPres} we present preservation results, in a very general context that covers many of the mentioned cardinal invariants, that allow to preserve some lower or upper bounds of these cardinals under forcing extensions. These results are fundamental to calculate the values of the invariants in the applications. In Section \ref{SecMatIt} we introduce the cases of matrix iterations we need for our applications and show how the preservation results of Section \ref{SecPres} are useful in these forcing constructions. These two sections summarize the theorical background presented in \cite{mejia} with the difference that we add Examples \ref{SubsecSplit} and \ref{SubsecFinSp}, prove a preservation result respect to Laver forcing with an ultrafilter (Lemma \ref{LaverStarPitchfork}) and extend the context of the construction of matrix iterations in Context \ref{ContextMatrix}.\\
Section \ref{SecAppl} contains the constructions of models where the mentioned cardinal invariants assume some arbitrary preassigned values, extending the same type of applications that are shown in \cite{mejia} to some other cardinal invariants that do not appear in Cichon's diagram.
In the last section, we mention some questions that, answered positively, imply interesting extensions of our examples.
\begin{acknowledgements}
  The author is very thankful with professor J. Brendle for all his guidance, constructive discussions and help with the final version of this text, especially for noting that $(+_{\Bor,\varpropto})$ does \underline{not} hold (see discussion after Example \ref{SubsecSplit}), teaching to the author Example \ref{SubsecFinSp} and its relation with $\Bor$ and $\Eor$ (Lemma \ref{PresPlusUnbd->PresRhd}) and for noting Lemma \ref{LaverStarPitchfork}(c). Thanks to these valuable inputs, Theorems \ref{dfrak,NonN,r-large}, \ref{dfrak-large}, \ref{MatrixCont-large}(b)(d)(f), \ref{MatrixCofN-large}(b)(d)(f) and \ref{MatrixNonN-large}(c) were possible to be included in the applications.\\
  The author is also grateful to professor T. Miyamoto for his invitation to such a great conference.
\end{acknowledgements}

\section{Preservation properties}\label{SecPres}

Throughout this section, let $\kappa$ be an uncountable regular cardinal and $\lambda\geq\kappa$ infinite cardinal. First, we present a general context that allows to describe some cardinal invariants and associate with it some preservation results. Second, we describe particular cases of this context that correspond to some of the cardinal invariants introduced in Section \ref{SecIntro} and list the preservation results that hold for each case. At the end, preservation results about unbounded reals are mentioned.
\begin{context}[\cite{gold}, {\cite[Section 6.4]{barju}}]\label{ContextUnbd}
 We fix an increasing sequence $\langle\sqsubset_n\rangle_{n<\omega}$ of 2-place relations in $\omega^\omega$ such that
\begin{itemize}
   \item each $\sqsubset_n$ ($n<\omega)$ is a closed relation (in the arithmetical sense) and
   \item for all $n<\omega$ and $g\in\omega^\omega$, $(\sqsubset_n)^g=\cj{f\in\omega^\omega}{f\sqsubset_n g}$ is (closed) n.w.d.
\end{itemize}
Put $\sqsubset=\bigcup_{n<\omega}\sqsubset_n$. Therefore, for every $g\in\omega^\omega$, $(\sqsubset)^g$ is an $F_\sigma$ meager set.\\
$F\subseteq\omega^\omega$ is a \emph{$\sqsubset$-unbounded family} if, for every $g\in\omega^\omega$, there exists an $f\in F$ such that
$f\not\sqsubset g$. We define the cardinal $\bfrak_\sqsubset$ as the least size of a $\sqsubset$-unbounded family. Besides,
$D\subseteq\omega^\omega$ is a \emph{$\sqsubset$-dominating family} if, for every $x\in\omega^\omega$, there exists an $f\in D$ such that
$x\sqsubset f$. Likewise, we define the cardinal $\dfrak_\sqsubset$ as the least size of a $\sqsubset$-dominating family.\\
Given a set $Y$, we say that a real $f\in\omega^\omega$ is \emph{$\sqsubset$-unbounded over $Y$} if $f\not\sqsubset g$ for every $g\in
Y\cap\omega^\omega$.
\end{context}
Although we define Context \ref{ContextUnbd} for $\omega^\omega$, we can use, in general, the same notion by changing the space for the domain or
the range of $\sqsubset$ to another uncountable Polish space, like $2^\omega$ or other spaces whose members can be coded by reals in
$\omega^\omega$.

\subsection{Preservation of $\sqsubset$-unbounded families}\label{SubsecPresUnbdFly}

\begin{definition}\label{DefTriangle}
  For a set $F\subseteq\omega^\omega$, the property
  $(\blacktriangle,\sqsubset,F,\kappa)$ holds if, for all $X\subseteq\omega^\omega$ such that $|X|<\kappa$,
             there exists an $f\in F$ which is $\sqsubset$-unbounded over $X$.
\end{definition}
This property implies directly that $F$ is a $\sqsubset$-unbounded family and that no set of size $<\kappa$ is $\sqsubset$-dominating, that is,
\begin{lemma}\label{TriangleImpl}
   $(\blacktriangle,\sqsubset,F,\kappa)$ implies that $\bfrak_\sqsubset\leq|F|$ and $\kappa\leq\dfrak_{\sqsubset}$.
\end{lemma}
\begin{definition}[{\cite[Def. 6.4.4.3]{barju}}]\label{DefPlusProp}
  For a forcing notion $\Por$, the property
  $(+^\kappa_{\Por,\sqsubset})$ holds if, for every $\Por$-name $\dot{h}$ of a real in $\omega^\omega$, there exists a set
        $Y\subseteq\omega^\omega$ such that $|Y|<\kappa$ and, for every $f\in\omega^\omega$, if $f$ is $\sqsubset$-unbounded over $Y$,
        then $\Vdash f\not\sqsubset\dot{h}$.
\end{definition}
When $\kappa=\aleph_1$, we just write $(+_{\Por,\sqsubset})$.\\
$(+^\kappa_{\Por,\sqsubset})$ is a standard property associated to the preservation of $\bfrak_\sqsubset\leq\kappa$ and the preservation of $\dfrak_\sqsubset$ large through forcing extensions of $\Por$.
\begin{lemma}[{\cite[Lemma 6.4.8]{barju}},{\cite[Lemma 3]{mejia}}]\label{PreserTriangle}
   Assume $(+^\kappa_{\Por,\sqsubset})$. Then, the statements $(\blacktriangle,\sqsubset,F,\kappa)$ and
   \lcom$\dfrak_\sqsubset\geq\lambda$\rcom\ are preserved in generic extensions of $\Por$.
\end{lemma}
The property $(+^\kappa_{\Por,\sqsubset})$ is preserved through fsi.
\begin{theorem}[Judah and Shelah, \cite{jushe}, {\cite[Thm. 6.4.12.2]{barju}}, \cite{brendle}]\label{preservPlus}
   Let $\kappa$ be an uncountable cardinal, $\Por_\delta=\langle\Por_\alpha,\Qnm_\alpha\rangle_{\alpha<\delta}$ a fsi of $\kappa$-cc forcing.
   If $\forall_{\alpha<\delta}\big(\Vdash_{\Por_\alpha}(+^\kappa_{\Qnm_\alpha,\sqsubset})\big)$, then $(+^\kappa_{\Por_\delta,\sqsubset})$.
\end{theorem}
Notice that, if $\Por$ and $\Qor$ are posets such that $\Por$ is completely embedded in $\Qor$, then
$(+^{\kappa}_{\Qor,\sqsubset})$ implies $(+^{\kappa}_{\Por,\sqsubset})$.

\subsection{Particular cases}\label{SubsecPartCase}

Before presenting the particular cases of $\bfrak_\sqsubset$ and $\dfrak_\sqsubset$ of our interest, we claim that the property $(+^\kappa_{\Por,\sqsubset})$ holds for small forcing notions.
\begin{lemma}[{\cite[Lemma 4]{mejia}}]\label{smallPlus}
   If $\Por$ is a poset and $|\Por|<\kappa$, then $(+^\kappa_{\Por,\sqsubset})$. In particular, $(+_{\Cor,\sqsubset})$ always holds.
\end{lemma}
\begin{example}[Preserving non-meager sets]\label{SubsecNonMeag}
  For $f,g\in\omega^\omega$, define $f\eqcirc_n g\sii\forall_{k\geq n}(f(k)\neq g(k))$, so $f\eqcirc g\sii\forall^\infty_{k\in\omega}(f(k)\neq
  g(k))$. From the characterization of covering and uniformity of category (see \cite[Thm. 2.4.1 and 2.4.7]{barju}), it follows that
  $\bfrak_{\eqcirc}=\non(\Mwf)$ and $\dfrak_{\eqcirc}=\cov(\Mwf)$.
\end{example}
\begin{example}[Preserving unbounded families]\label{SubsecUnbd}
   For $f,g\in\omega^\omega$, define $f<^*_n g\sii\forall_{k\geq n}(f(k)<g(k))$, so $f<^*g\sii\forall^\infty_{k\in\omega}(f(k)<g(k))$. Clearly,
   $\bfrak_{<^*}=\bfrak$ and $\dfrak_{<^*}=\dfrak$. $(+_{\Bor,<^*})$ holds because $\Bor$ is $\omega^\omega$-bounding, also
\end{example}
\begin{lemma}[Miller, \cite{miller}]\label{EvDiffPlus}
    $(+_{\Eor,<^*})$ holds.
\end{lemma}
\begin{example}[Preserving null-covering families]\label{SubsecCovNull}
   Fix, from now on, $\langle
   I_n\rangle_{n<\omega}$ an interval partition of $\omega$ (see Example \ref{SubsecFinSp} for a definition of this) such that $\forall_{n<\omega}(|I_n|=2^{n+1})$ . For $f,g\in2^\omega$ define
   $f\pitchfork_ng\sii\forall_{k\geq n}(f\frestr I_k\neq g\frestr I_k)$, so $f\pitchfork g\sii \forall^\infty_{k<\omega}(f\frestr I_k\neq g\frestr
   I_k)$. Clearly, $(\pitchfork)^g$ is a co-null $F_\sigma$ meager set.
\end{example}
\begin{lemma}[{\cite[Lemma $1^*$]{brendle}}]\label{centeredFork}
   Given $\mu<\kappa$ infinite cardinal, every $\mu$-centered forcing notion satisfies $(+^\kappa_{\cdot,\pitchfork})$.
\end{lemma}
The following result shows why $\pitchfork$ is useful to deal with preserving $\cov(\Nwf)$ small and $\non(\Nwf)$ large.
\begin{lemma}[{\cite[Lemma 7]{mejia}}]\label{InvforPitchfork}
   $\cov(\Nwf)\leq\bfrak_\pitchfork\leq\non(\Mwf)$ and $\cov(\Mwf)\leq\dfrak_\pitchfork\leq\non(\Nwf)$.
\end{lemma}
\begin{example}[Preserving union of null sets is not null]\label{SubsecAddNull}
   Define
   \[\Sl=\cj{\varphi:\omega\to[\omega]^{<\omega}}{\exists_{k<\omega}\forall_{n<\omega}(|\varphi(n)|\leq(n+1)^k)}\]
   the \emph{space of slaloms}. As a Polish space, this is coded by reals in $\omega^\omega$. For $f\in\omega^\omega$ and a slalom $\varphi$,
   define $f\in^*_n\varphi\sii\forall_{k\geq n}(f(k)\in\varphi(k))$, so\footnote{In \cite{mejia}, the relation $\in^*$ is denoted by $\subseteq^*$, but it may be confused with the relation of `almost containment' between subsets of $\omega$.} $f\in^*\varphi\sii \forall_{k<\omega}^\infty(f(k)\in
   \varphi(k))$. From the characterization given by \cite[Thm. 2.3.9]{barju}, $\bfrak_{\in^*}=\add(\Nwf)$ and
   $\dfrak_{\in^*}=\cof(\Nwf)$.
\end{example}
\begin{lemma}[Judah and Shelah, \cite{jushe} and \cite{brendle}]\label{centeredAddNull}
   Given $\mu<\kappa$ infinite cardinals, every $\mu$-centered forcing notion satisfies $(+^\kappa_{\cdot,\in^*})$.
\end{lemma}
\begin{lemma}[Kamburelis, \cite{kamburelis}]\label{spamPlus}
   Every boolean algebra with a strictly positive finitely additive measure (see \cite{kamburelis} for this concept)
   satisfies $(+_{\cdot,\in^*})$. In particular, subalgebras of the random
   algebra satisfy that property.
\end{lemma}
\begin{example}[Preserving splitting families]\label{SubsecSplit}
   For $A,B\in[\omega]^\omega$, define $A\varpropto_n B\sii(B\menos n\subseteq A\disy B\menos n\subseteq\omega\menos A)$, so
   $A\varpropto B\sii(B\subseteq^* A\disy B\subseteq^*\omega\menos A)$. Note also that $A\not\varpropto B$ iff \emph{$A$ splits $B$}, that is, $A\cap B$ and $B\menos A$ are infinite. It is clear from the standard definitions that the \emph{splitting number} is $\sfrak=\bfrak_\varpropto$ and the \emph{reaping number} is $\rfrak=\dfrak_\varpropto$.
\end{example}
\begin{lemma}[Baumgartner and Dordal, \cite{baudor} and {\cite[Main Lemma 3.8]{brendlebog}}]\label{HechPlus}
   $(+_{\Dor,\varpropto})$ holds.
\end{lemma}
Recall that $\Bor$ is given by the complete boolean algebra of Borel sets of $2^\omega$ modulo the $\sigma$-ideal $\Nwf$. Let $\mu$ be the Lebesgue-measure corresponding to $2^\omega$ and, for a formula $\psi$ in the forcing language of $\Bor$,
$||\psi||$ denotes the supremum of the conditions in $\Bor$ that forces $\psi$. If there exists such condition, $||\psi||$ becomes the maximum one. Recall the interval partition $\langle I_n\rangle_{n<\omega}$ fixed in example \ref{SubsecCovNull}.\\
Note that $(+_{\Bor,\varpropto})$ does \underline{not} hold.
Indeed, define a $\Bor$-name $\dot{x}$ for an infinite subset of $\omega$ such that $\mu(||k\in\dot{x}||)=1/{2^{n+1}}$ for any $k\in I_n$. Given any sequence $\{z_n\}_{n\in\omega}$ of infinite subsets of $\omega$, it is easy to construct an $a\subseteq\omega\menos I_0$ infinite such that, for any $n<\omega$, $a$ splits $z_n$ and $|a\cap I_n|\leq 1$. Note that $\mu(||\dot{x}\cap a\neq\varnothing||)\leq1/2$, moreover, $\Vdash_\Bor|\dot{x}\cap a|<\aleph_0$.\\
Therefore, any poset that adds random reals does \underline{not} satisfy $(+_{\cdot,\varpropto})$. In particular, $(+_{\Aor,\varpropto})$ does \underline{not} hold.\\
It is not known whether $(+_{\Eor,\varpropto})$ holds.
\begin{example}[Preserving finitely splitting families]\label{SubsecFinSp}
   Say that $\bar{J}=\langle J_n\rangle_{n<\omega}$ is an \emph{interval partition of $\omega$} if it is a partition of $\omega$ into non-empty finite intervals such that $\max(J_n)<\min(J_{n+1})$ for all $n<\omega$. For $a\in[\omega]^\omega$ and an interval partition $\bar{J}$ of $\omega$, define $a\rhd_n\bar{J}\sii(\forall_{k\geq n}(J_k\nsubseteq a)\disy\forall_{k\geq n}(J_k\nsubseteq\omega\menos a))$, so $a\rhd\bar{J}\sii(\forall^\infty_{k\in\omega}(I_n\nsubseteq a)\disy\forall^\infty_{k\in\omega}(I_n\nsubseteq \omega\menos a))$. $a\ntriangleright\bar{J}$ is known as \emph{$a$ splits $\bar{J}$}, $\ffrak_\sfrak=\bfrak_\rhd=\max\{\bfrak,\sfrak\}$ is the \emph{finitely splitting number} and $\ffrak_\rfrak=\dfrak_\rhd=\min\{\dfrak,\rfrak\}$ is the \emph{finitely reaping number}. See \cite{kamb} for details about these cardinal invariants.
\end{example}
\begin{lemma}\label{PresPlusUnbd->PresRhd}
   For a poset $\Por$, $(+_{\Por,<^*})$ implies $(+_{\Por,\rhd})$. In particular, $(+_{\cdot,\rhd})$ holds for $\Bor$ and $\Eor$.
\end{lemma}
\begin{proof}
   Let $\dot{\bar{J}}$ be a $\Por$-name of an interval partition of $\omega$. By $(+_{\Por,<^*})$, let
   $\{h_n\}_{n<\omega}$ be a sequence of reals in $\omega^\omega$ such that $\Vdash\exists^\infty_n(\max(\dot{J}_n)+1\leq f(n))$ for any $f\in\omega^\omega$ which is $<^*$-unbounded over $\{h_n\}_{n<\omega}$. Choose an $h\in\omega^\omega$ which is a strictly increasing upper $<^*$-bound of $\{h_n\}_{n<\omega}$ such that $h(0)>0$. Define $\hat{h}\in\omega^\omega$ recursively, where
   $\hat{h}(0)=0$ and $\hat{h}(n+1)=h(\hat{h}(n))$. Put $J'_n:=[\hat{h}(2n),\hat{h}(2n+2))$ (interval notation), so
   $\bar{J}':=\langle J'_n\rangle_{n\in\omega}$ is an interval partition of $\omega$. It is enough to prove $\Vdash a\ntriangleright\dot{\bar{J}}$ for any $a\in[\omega]^\omega$ such that $a\ntriangleright\bar{J'}$. Indeed, define $f\in\omega^\omega$ such that
   \[f(n)=\left\{\begin{array}{ll}
       h(n) & \textrm{$n\in[\hat{h}(2k),\hat{h}(2k+1))$ and $J'_k\subseteq a$ for some $k\in\omega$,}\\
       0    & \textrm{otherwise.}
   \end{array}\right.\]
   It is clear that $f\not<^*h$, so $\Vdash\exists^\infty_n(\max(\dot{J}_n)+1\leq f(n))$. Now, let $G$ be a $\Por$-generic set over the ground model. In $V[G]$: fix $m<\omega$ and choose $n,k'\in\omega$ such that
   $\hat{h}(k')>m$, $n\in[\hat{h}(k'),\hat{h}(k'+1))$ and $\max(J_n)+1\leq f(n)$. As $f(n)$ cannot be 0, then
   $k'=2k$ for some $k\in\omega$, $J'_k\subseteq a$ and $f(n)=h(n)$. It is easy to check that $J_n\subseteq[n,f(n))\subseteq J'_k\subseteq a$. This gives us $\exists^\infty_n(J_n\subseteq a)$. To get $\exists^\infty_n(J_n\subseteq\omega\menos a)$, do the same argument but change $a$ by $\omega\menos a$ in the definition of $f$.
\end{proof}

\subsection{Preservation of $\sqsubset$-unbounded reals}

For the rest of this section, fix $M\subseteq N$ models of $\thzfc$, $\sqsubset$ a relation as in Context \ref{ContextUnbd} and $c\in N\cap\omega^\omega$ a $\sqsubset$-unbounded real over $M$.
\begin{definition}\label{DefCompSubordM}
   Given $\Por\in M$ and $\Qor$ posets, we say that \emph{$\Por$ is a complete suborder of $\Qor$ with respect to $M$}, denoted by
   $\Por\preceq_M\Qor$, if $\Por\subseteq\Qor$ and all maximal antichains of $\Por$ in $M$ are maximal antichains of $\Qor$.
\end{definition}
The main consequence of this definition is that, whenever $\Por\in M$ and $\Qor\in N$ are posets such that $\Por\preceq_M\Qor$ then, whenever
$G$ is $\Qor$-generic over $N$, $\Por\cap G$ is a $\Por$-generic set over $M$. Here, we are interested in the case where the real $c$ can be
preserved to be $\sqsubset$-unbounded over $M[G\cap\Por]$.
\begin{definition}\label{DefPresUnbdg}
   Assume $\Por\in M$ and $\Qor\in N$ posets such that $\Por\preceq_M\Qor$. We say that the property $(\star,\Por,\Qor,M,N,\sqsubset,c)$
   holds iff, for every $\dot{h}\in M$ $\Por$-name for a real in $\omega^\omega$, $\Vdash_{\Qor,N}c\not\sqsubset\dot{h}$.
   This is equivalent to saying that $\Vdash_{\Qor,N}$\lcom$c$ is $\sqsubset$-unbounded over $M^{\Por}$\rcom, that is,
   $c$ is $\sqsubset$-unbounded over $M[G\cap\Por]$ for every $G$ $\Qor$-generic over $N$.
\end{definition}
The last two definitions are important notions used in \cite{blsh}, \cite{BF} and \cite{mejia} for the preservation of unbounded reals and the
construction of matrix iterations. The following result is the first example of this preservation property that has been used for a matrix
iteration construction. It has been proved for $<^*$ but a proof for $\in^*$ can be done by a similar argument.
\begin{lemma}[Blass and Shelah, {\cite[Main Lemma]{blsh}}]\label{mathiasStar}
   Let $\sqsubset$ be $<^*$ or $\in^*$. In $M$, let $\Uwf$ be a non-principal ultrafilter on $\omega$. If $c\in N$ is a $\sqsubset$-unbounded real over $M$, then there exists an ultrafilter $\Vwf$ in $N$ extending $\Uwf$ such that $\Mor_\Uwf\preceq_M\Mor_\Vwf$ and $(\star,\Mor_\Uwf,\Mor_\Vwf,M,N,\sqsubset,c)$ holds.
\end{lemma}
We don't know whether the foregoing Lemma holds for $\sqsubset=\pitchfork$ (see Question \ref{mathiasStarPitchfork}), but we can prove a version for Laver forcing with an ultrafilter. Given a filter $\Fwf$ on $[\omega]^\omega$ that contains the cofinite subsets of $\omega$, \emph{Laver forcing with $\Fwf$} is the poset $\Lor_\Fwf$ whose conditions are infinitely-branching subtrees $T$ of $\omega^\omega$ such that $\cj{i<\omega}{\sigma\widehat{\ \ }\langle i\rangle\in T}\in\Fwf$ for any $\sigma\in T$ such that $\sigma\supseteq\mathrm{stem}(T)$, where $\mathrm{stem}(T)$, the \emph{stem of $T$}, is the unique branching node of $T$ of minimal level. The order of $\Lor_\Fwf$ is $\subseteq$. It is well known that this forcing notion is $\sigma$-centered and that adds a dominating real $l_\Fwf$ over the ground model such that $\ran(l_\Fwf)$ is a pseudo-intersection of $\Fwf$.
\begin{lemma}[{\cite[Thm. 9]{blass02}} Pure decision property]\label{PureDec}
   Let $\Uwf$ be a non-principal ultrafilter on $\omega$, $s\in\omega^{<\omega}$ and $\psi$ a formula in the forcing language of $\Lor_\Uwf$. Then, there exists a $T\in\Lor_\Uwf$ such that $\mathrm{stem}(T)=s$ and, either $T\Vdash\psi$ or $T\Vdash\neg\psi$.
\end{lemma}
\begin{lemma}\label{LaverStarPitchfork}
   In $M$, let $\Uwf$ be a non-principal ultrafilter on $\omega$ and, in $N$, let $\Vwf$ be a non-principal ultrafilter on $\omega$ containing $\Uwf$. Then,
   \begin{enumerate}[(a)]
     \item \emph{(Shelah \cite{shelah}, see also {\cite[Lemma 2.1]{brendle03}} and
           {\cite[Lemma 8]{brendle02}})} $\Lor_\Uwf\preceq_M\Lor_\Vwf$.
     \item Let $F\subseteq M$ finite and $\dot{x}\in M$ a $\Lor_\Uwf$-name for a member of $F$. For $s\in\omega^{<\omega}$ let $z_s\in F$ be such that $T\nVdash_{\Lor_\Uwf,M}\dot{x}\neq z_s$ for any $T\in\Lor_\Uwf$ with $\mathrm{stem}(T)=s$. Then, $T\nVdash_{\Lor_\Vwf,N}\dot{x}\neq z_s$ for any $T\in\Lor_\Vwf$ with $\mathrm{stem}(T)=s$.
     \item $(\star,\Lor_\Uwf,\Lor_\Vwf,M,N,\pitchfork,c)$ holds for any $c\in2^\omega$ $\pitchfork$-unbounded over $M$.
   \end{enumerate}
\end{lemma}
\begin{proof}
  \begin{enumerate}[(a)]
  \setcounter{enumi}{1}
      \item In $N$, let $T\in\Lor_\Vwf$ with $\mathrm{stem}(T)=s$. In $M$ find, by Lemma \ref{PureDec}, $T'\in\Lor_\Uwf$ with $\mathrm{stem}(T')=s$ that decides the formula \lcom$\dot{x}=z_s$\rcom. By hypothesis, it is clear that $T'\Vdash_{\Lor_\Uwf,M}\dot{x}=z_s$. Now, in $N$, it is easy to see that $T'\Vdash_{\Lor_\Vwf,N}\dot{x}=z_s$ and, as $T'$ and $T$ have the same stem, they are compatible in $\Lor_\Vwf$, so $T\cap T'\Vdash_{\Lor_\Vwf,N}\dot{x}=z_s$.
      \item The idea of this proof is taken from \cite[Lemma $1^*$]{brendle} (see Lemma \ref{centeredFork}). Let $\dot{x}\in M$ be a $\Lor_\Uwf$-name for a real in $2^\omega$. For each $n<\omega$ and $s\in\omega^{<\omega}$, find a $\sigma_{s,n}\in 2^{I_n}$ such that
          $T\nVdash_{\Lor_\Uwf,M}\dot{x}\frestr I_n\neq\sigma_{s,n}$ for any $T\in\Lor_\Uwf\cap M$ with $\mathrm{stem}(T)=s$ (this can be found because $\Lor_\Uwf$ is $\sigma$-centered). Now, in $N$, $T\nVdash_{\Lor_\Vwf,N}\dot{x}\frestr I_n\neq\sigma_{s,n}$ for any $T\in\Lor_\Vwf$ with $\mathrm{stem}(T)=s$. Let $x_s=\bigcup_{n<\omega}\sigma_{s,n}\in2^\omega\cap M$, so $c\not\pitchfork x_s$ for any $s\in\omega^{<\omega}$. It is easy to see that $\Vdash_{\Lor_\Vwf,N}c\not\pitchfork\dot{x}$.
  \end{enumerate}
\end{proof}

In relation with the preservation property of Definition \ref{DefPlusProp}, we have the following.
\begin{lemma}[{\cite[Thm. 7]{mejia}}]\label{suslinStar}
   Let $\Por$ be a Suslin ccc forcing notion with parameters in $M$. If $(+_{\Por,\sqsubset})$ holds in $M$,
   then $(\star,\Por^M,\Por^N,M,N,\sqsubset,c)$ holds.
\end{lemma}
As a last example, we have
\begin{lemma}[Brendle and Fischer, {\cite[Lemma 11]{BF}}]\label{fixedStar}
   For a forcing notion $\Por\in M$, $(\star,\Por,\Por,M,N,\sqsubset,c)$ holds.
\end{lemma}
Finally, unbounded reals are preserved in fsi, as established by these last two results.
\begin{lemma}[{\cite[Lemmas 10 and 13]{BF}}]\label{CompSubordIt}
   Let $\delta$ be an ordinal in $M$, $\Por_{0,\delta}=\langle\Por_{0,\alpha},\Qnm_{0,\alpha}\rangle_{\alpha<\delta}$
     a fsi of posets defined in $M$ and $\Por_{1,\delta}=\langle\Por_{1,\alpha},\Qnm_{1,\alpha}\rangle_{\alpha<\delta}$
     a fsi of posets defined in $N$. Then, $\Por_{0,\delta}\preceq_M\Por_{1,\delta}$ iff, for every $\alpha<\delta$,
     $\Vdash_{\Por_{1,\alpha},N}\Qnm_{0,\alpha}\preceq_{M^{\Por_{0,\alpha}}}\Qnm_{1,\alpha}$.
\end{lemma}
\begin{theorem}[Blass and Shelah, \cite{blsh}, {\cite[Lemma 12]{BF}}]\label{PreservStar}
   With the notation in Lemma \ref{CompSubordIt}, assume that $\Por_{0,\delta}\preceq_M\Por_{1,\delta}$. Then,
   $(\star,\Por_{0,\delta},\Por_{1,\delta},M,N,\sqsubset,c)$ holds iff, for every $\alpha<\delta$,
   \[\Vdash_{\Por_{1,\alpha},N}(\star,\Qnm_{0,\alpha},\Qnm_{1,\alpha},M^{\Por_{0,\alpha}},N^{\Por_{1,\alpha}},\sqsubset,c).\]
\end{theorem}

\section{Matrix iterations}\label{SecMatIt}

Throughout this section, we work in a model $V$ of $\thzfc$. Fix two ordinals $\delta$ and $\gamma$.
\begin{definition}[Blass and Shelah, \cite{blsh}, \cite{BF} and \cite{mejia}]\label{DefMatrixIt}
A \emph{matrix iteration of ccc posets} is given by $\Por_{\delta,\gamma}=\langle\langle\Por_{\alpha,\xi},\Qnm_{\alpha,\xi}
\rangle_{\xi<\gamma}\rangle_{\alpha\leq\delta}$ with the following conditions.
\begin{enumerate}[(1)]
   \item $\Por_{\delta,0}=\langle\Por_{\alpha,0},\Rnm_{\alpha}\rangle_{\alpha<\delta}$ is a fsi of ccc posets.
   \item For all $\alpha\leq\delta$, $\langle\Por_{\alpha,\xi},\Qnm_{\alpha,\xi}
         \rangle_{\xi<\gamma}$ is a fsi of ccc posets
   \item For all $\xi<\gamma$ and $\alpha<\beta\leq\delta$,
         $\Vdash_{\Por_{\beta,\xi}}\Qnm_{\alpha,\xi}\preceq_{V^{\Por_{\alpha,\xi}}}\Qnm_{\beta,\xi}$.
\end{enumerate}
By Lemma \ref{CompSubordIt}, condition (3) is equivalent to saying that $\Por_{\alpha,\xi}$ is a complete suborder of $\Por_{\beta,\gamma}$
for every $\alpha<\beta\leq\delta$ and $\xi\leq\gamma$.\\
In the context of matrix iterations, when $\alpha\leq\delta$, $\xi\leq\gamma$ and $G_{\alpha,\xi}$ is $\Por_{\alpha,\xi}$-generic over $V$, we
denote
$V_{\alpha,\xi}=V[G_{\alpha,\xi}]$. Note that $V_{0,0}=V$.\\
Figure \ref{fig:2} shows the form in which we think of a matrix iteration. The iteration defined in (1) is represented by the leftmost vertical
iteration and, at each $\alpha$-stage of this iteration ($\alpha\leq\delta$), a horizontal iteration is performed as it is represented in (2).
\begin{figure}
\begin{center}
  \includegraphics[scale=0.48]{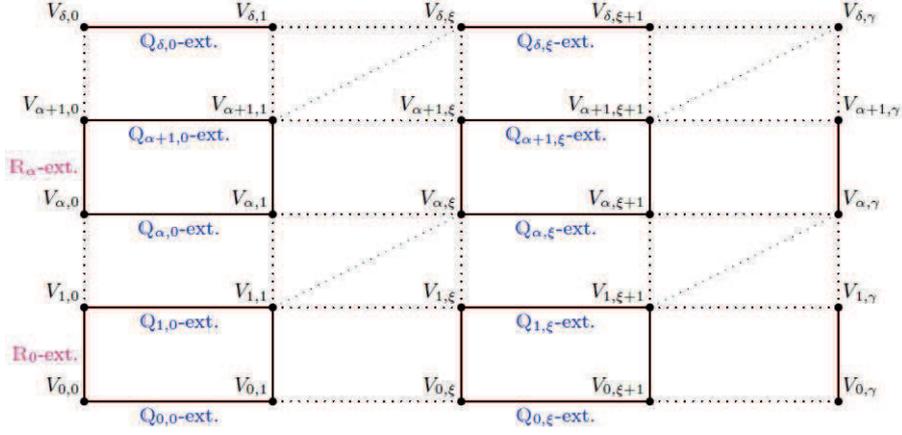}
\caption{Matrix iteration} \label{fig:2}
\end{center}
\end{figure}
\end{definition}
The construction of the matrix iterations for the models in Section \ref{SecAppl} corresponds to the following particular case, which we
fix from now on.
\begin{context}\label{ContextMatrix}
   Put $\gamma=S\cup U\cup L\cup T$ as a disjoint union and fix a function $\Delta:T\to\delta$. For $\xi\in S$ fix $\Sor_\xi$ a Suslin ccc poset with parameters in $V$.
   Define the matrix iteration $\Por_{\delta,\gamma}=\langle\langle\Por_{\alpha,\xi},\Qnm_{\alpha,\xi}
   \rangle_{\xi<\gamma}\rangle_{\alpha\leq\delta}$ as follows.
   \begin{enumerate}[(1)]
       \item $\Por_{\delta,0}=\langle\Por_{\alpha,0},\Cnm\rangle_{\alpha<\delta}$ (fsi of Cohen forcing).
       \item For a fixed $\xi<\gamma$, $\Qnm_{\alpha,\xi}$ is defined for all $\alpha\leq\delta$ according to one of the following cases.
   \begin{enumerate}[(i)]
       \item For $\xi\in S$, $\Qnm_{\alpha,\xi}=\Snm_\xi$ as a $\Por_{\alpha,\xi}$-name of $\Sor_\xi$.
       \item For $\xi\in U$ define, by recursion on $\alpha\leq\delta$, a $\Por_{\alpha,\xi}$-name $\dot{\Uwf}_{\alpha,\xi}$
             of a non-principal ultrafilter on $\omega$ such that
             \begin{itemize}
                \item For $\alpha<\beta\leq\delta$, $\Vdash_{\beta,\xi}\textrm{\lcom} \dot{\Uwf}_{\alpha,\xi}\subseteq\dot{\Uwf}_{\beta,\xi}\textrm{\ and\ }
                      \Mor_{\dot{\Uwf}_{\alpha,\xi}}\preceq_{V_{\alpha,\xi}}
                      \Mor_{\dot{\Uwf}_{\beta,\xi}}\textrm{\rcom}$.
                \item It is forced by $\Por_{\alpha,\xi}$ that $\dot{\Uwf}_{\alpha,\xi}$ contains the Mathias reals added by
                      $\Por_{\alpha,\xi'+1}$ for each $\xi'\in U$, $\xi'<\xi$.
                \item $\dot{\Uwf}_{\alpha+1,\xi}$ comes from the application of Lemma \ref{mathiasStar} to $M=V_{\alpha,\xi}$, $N=V_{\alpha+1,\xi}$, $\Uwf=\Uwf_{\alpha,\xi}$ ($\dot{\Uwf}_{\alpha,\xi}$ as interpreted in the $\Por_{\alpha,\xi}$-extension)
                    and \emph{$\sqsubset$ is fixed from the beginning of the construction of the matrix iteration.}
             \end{itemize}
             The details of how this construction can be done can be found in \cite{blsh} and in Sections 4 and 5 of \cite{BF}. Put $\Qnm_{\alpha,\xi}=\Mor_{\dot{\Uwf}_{\alpha,\xi}}$.
       \item For $\xi\in L$ define, by recursion on $\alpha\leq\delta$, a $\Por_{\alpha,\xi}$-name $\dot{\Uwf}_{\alpha,\xi}$
             of a non-principal ultrafilter on $\omega$ such that
             \begin{itemize}
                \item For $\alpha<\beta\leq\delta$, $\Vdash_{\beta,\xi} \dot{\Uwf}_{\alpha,\xi}\subseteq\dot{\Uwf}_{\beta,\xi}$.
                \item It is forced by $\Por_{\alpha,\xi}$ that $\dot{\Uwf}_{\alpha,\xi}$ contains the range of the Laver reals added by
                      $\Por_{\alpha,\xi'+1}$ for each $\xi'\in L$, $\xi'<\xi$.
                \item If $\beta\leq\delta$ has uncountable cofinality, then $\Vdash_{\beta,\xi}\Uwf_{\beta,\xi}=\bigcup_{\alpha<\beta}\Uwf_{\alpha,\xi}$
             \end{itemize}
             Put $\Qnm_{\alpha,\xi}=\Lor_{\dot{\Uwf}_{\alpha,\xi}}$ (Lemma \ref{LaverStarPitchfork}
             is relevant for this construction).
       \item For $\xi\in T$ fix a $\Por_{\Delta(\xi),\xi}$-name $\Tnm_\xi$ of a ccc poset whose conditions are reals. Put
           \[\Qnm_{\alpha,\xi}=\left\{
               \begin{array}{ll}
                   \mathds{1} & \textrm{if $\alpha\leq \Delta(\xi)$,}\\
                   \Tnm_\xi & \textrm{if $\alpha>\Delta(\xi)$.}
               \end{array}
               \right.\]
   \end{enumerate}
   \end{enumerate}
   It is clear that this satisfies the conditions of the Definition \ref{DefMatrixIt}.\\
   From the iteration in (1), for $\alpha<\delta$ let $\dot{c}_\alpha$ be a $\Por_{\alpha+1,0}$-name
   for a Cohen real over $V_{\alpha,0}$. Therefore, from Context \ref{ContextUnbd} it is clear that $\dot{c}_\alpha$ represents a
   $\sqsubset$-unbounded real over $V_{\alpha,0}$ (actually, this is the only place where we use in this paper the second condition of
   Context \ref{ContextUnbd}).
\end{context}
The same argument as in the proof of \cite[Lemma 15]{BF} yields the following.
\begin{theorem}[Brendle and Fischer]\label{realsmatrixbelow}
   Assume that $\delta$ has uncountable cofinality and $\xi\leq\gamma$.
   \begin{enumerate}[(a)]
      \item If $p\in\Por_{\delta,\xi}$ then there exists an $\alpha<\delta$ such that $p\in\Por_{\alpha,\xi}$.
      \item If $\dot{h}$ is a $\Por_{\delta,\xi}$-name for a real, then there exists an $\alpha<\delta$ such that
            $\dot{h}$ is a $\Por_{\alpha,\xi}$-name.
   \end{enumerate}
\end{theorem}
When we go through generic extensions of the matrix iteration, for every $\alpha<\delta$ we are interested in preserving the
$\sqsubset$-unboundedness of $\dot{c}_\alpha$ through the horizontal iterations. The following results state conditions that guarantee this.
\begin{theorem}[{\cite[Thm. 10]{mejia}}]\label{PresvUnbdg}
   Assume that
   \begin{enumerate}[(i)]
      \item For every $\xi\in S$ and $\alpha\leq\delta$, $\Vdash_{\Por_{\alpha,\xi}}(+_{\Qnm_{\alpha,\xi},\sqsubset})$.
      \item If $U\neq\varnothing$, then $L=\varnothing$ and $\sqsubset$ is $<^*$ or $\in^*$ and is the one fixed for (2)(ii) in Context \ref{ContextMatrix}.
      \item If $L\neq\varnothing$, then $U=\varnothing$ and $\sqsubset$ is $\pitchfork$.
   \end{enumerate}
   Then, for all $\alpha<\delta$, $\Por_{\alpha+1,\gamma}$ forces that $\dot{c}_\alpha$
   is a $\sqsubset$-unbounded real over $V_{\alpha,\gamma}$.
\end{theorem}
\begin{corollary}[{\cite[Cor. 1]{mejia}}]\label{MainMatrixConsq}
   With the same assumptions as in Theorem \ref{PresvUnbdg}, if $\delta$ has uncountable cofinality, then
   $\Vdash_{\Por_{\delta,\gamma}}\cf(\delta)\leq\dfrak_\sqsubset$.
\end{corollary}
By Lemma \ref{InvforPitchfork}, Corollary \ref{MainMatrixConsq} holds for $\sqsubset=\pitchfork$ with $\non(\Nwf)$ in place of
$\dfrak_{\pitchfork}$.

\section{Applications}\label{SecAppl}

For any infinite cardinal $\lambda$, we use the notation
\begin{description}
  \item[$\mathbf{GCH}_\lambda$] For any infinite cardinal $\mu$,
       \[2^\mu=\left\{
            \begin{array}{ll}
               \lambda & \textrm{if $\mu<\cf(\lambda)$,}\\
               \lambda^+ & \textrm{if $\cf(\lambda)\leq\mu<\lambda$,}\\
               \mu^+ & \textrm{if $\lambda\leq\mu$.}
            \end{array}\right.\]
\end{description}
Throughout this section our results are given for a model $V$ of $\thzfc$. There, we fix $\mu_1\leq\mu_2\leq\mu_3\leq\nu\leq\kappa$
uncountable regular cardinals and a cardinal $\lambda\geq\kappa$. By using the same techniques as in \cite{brendle} and \cite[Sect. 3]{mejia}, we get the following three results.
\begin{theorem}\label{CovM-large}
    Assume, in $V$, $\mathbf{GCH}$ and $\cf(\lambda)\geq\mu_3$. Then, there exists a ccc poset that forces $\mathbf{GCH}_\lambda$,
   $\add(\Nwf)=\mu_1$, $\cov(\Nwf)=\mu_2$, $\pfrak=\non(\Mwf)=\mu_3$ and $\cov(\Mwf)=\cfrak=\lambda$.
\end{theorem}
\begin{proof}
    In the proof of \cite[Thm. 2]{mejia} it is constructed a model $V^3$, which is a generic extension of a ccc poset, that satisfies $\add(\Nwf)=\mu_1$, $\cov(\Nwf)=\mu_2$, $\add(\Mwf)=\cfrak=\mu_3$, $\mathbf{GCH}_{\mu_3}$, $(\blacktriangle,\in^*,A,\mu_1)$ and $(\blacktriangle,\pitchfork,B,\mu_2)$ where $A$ is some subset of $\omega^\omega$ of size $\mu_1$ and $B$ is some subset of $2^\omega$ of size $\mu_2$. Note that $(\blacktriangle,\eqcirc,C,\mu_3)$ holds for $C:=\omega^\omega\cap V^3$, which has size $\mu_3$.\\
    In $V^3$, perform a fsi $\langle\Por^3_\alpha,\Qnm^3_\alpha\rangle_{\alpha<\lambda}$ such that
    \begin{itemize}
       \item for $\alpha\equiv0\modulo3$, $\Qnm^3_\alpha$ is a $\Por^3_\alpha$-name for a subalgebra of $\Aor$ of size $<\mu_1$,
       \item for $\alpha\equiv1\modulo3$, $\Qnm^3_\alpha$ is a $\Por^3_\alpha$-name for a subalgebra of $\Bor$ of size $<\mu_2$ and
       \item for $\alpha\equiv2\modulo3$, $\Qnm^3_\alpha=\Mor_{\dot{\Fwf}_\alpha}$ where $\dot{\Fwf}_\alpha$ is a $\Por^3_{\alpha}$-name for a filter base of size $<\mu_3$.
    \end{itemize}
    By a book-keeping argument, we ensure to use all such subalgebras and filter bases. Like in the argument of the cited proof, in a extension $V^4$ of this iteration, $\mathbf{GCH}_\lambda$ holds, $\add(\Nwf)=\mu_1$, $\cov(\Nwf)=\mu_2$ and $\non(\Mwf)\leq\mu_3$ are preserved and $\cov(\Mwf)=\cfrak=\lambda$. $\pfrak\geq\mu_3$ because any filter base
    of size $<\mu_3$ has a pseudo-intersection, which is a Mathias real added at some step $\alpha<\lambda$ with $\alpha\equiv2\modulo 3$.
\end{proof}
\begin{theorem}\label{NonN,r-large}
   In $V$, assume $\mathbf{GCH}$ and $\cf(\lambda)\geq\mu_3$. Then, there exists a ccc poset that forces $\mathbf{GCH}_\lambda$,
   $\add(\Nwf)=\mu_1$, $\cov(\Nwf)=\mu_2$, $\pfrak=\sfrak=\mu_3$, $\add(\Mwf)=\cof(\Mwf)=\kappa$ and $\non(\Nwf)=\rfrak=\cfrak=\lambda$.
\end{theorem}
\begin{proof}
   Note that, in the proof of Theorem \ref{CovM-large}, $C'=[\omega]^\omega\cap V^3$ has size $\mu_3$ and $(\blacktriangle,\varpropto,C',\mu_3)$ holds in $V^3$. Also, by Lemma \ref{PreserTriangle}, this property is preserved in the model $V^4$ of the proof of Theorem \ref{CovM-large} because $(+^{\mu_3}_{\Por^3_\lambda,\varpropto})$ holds in $V^3$ from Lemma \ref{smallPlus} and Theorem \ref{preservPlus}. Now, in $V^4$, perform a fsi $\langle\Por^4_\alpha,\Qnm^4_\alpha\rangle_{\alpha<\kappa}$ such that
   \begin{itemize}
      \item for $\alpha\equiv0\modulo4$, $\Qnm^4_\alpha=\Dnm$ ($\Por^4_\alpha$-name for $\Dor$),
      \item for $\alpha\equiv1\modulo4$, $\Qnm^4_\alpha$ is the $\Por^4_\alpha$-name for the fsp (finite support product) of size $\lambda$
            of \emph{all} the subalgebras of $\Aor$ of size $<\mu_1$ in any $\Por_\alpha$-generic extension of $V^4$,
      \item for $\alpha\equiv2\modulo4$, $\Qnm^4_\alpha$ is the $\Por^4_\alpha$-name for fsp of size $\lambda$
            of \emph{all} the subalgebras of $\Bor$ of size $<\mu_2$ in any $\Por_\alpha$-generic extension of $V^4$, and
      \item for $\alpha\equiv3\modulo4$, $\Qnm^4_\alpha$ is the $\Por^4_\alpha$-name for fsp of size $\lambda$
            of $\Mor_\Fwf$ for \emph{all} the filter bases $\Fwf$ of size $<\mu_3$ in any $\Por^4_\alpha$-generic extension of $V^4$.
   \end{itemize}
   The results of Subsections \ref{SubsecPresUnbdFly} and \ref{SubsecPartCase} imply that $(+^{\mu_1}_{\Por^4_\kappa,\in^*})$, $(+^{\mu_2}_{\Por^4_\kappa,\pitchfork})$
   and $(+^{\mu_3}_{\Por^4_\kappa,\varpropto})$ hold. Let $V^5$ be a generic extension of this iteration. Then, in $V^5$, $\add(\Nwf)=\mu_1$, $\cov(\Nwf)=\mu_2$ and $\mu_3\leq\pfrak$ by a similar argument as in Theorem \ref{CovM-large} (and its corresponding cited result). It is clear that $\mathbf{GCH}_\lambda$ holds and $\cfrak\leq\lambda$. As $(\blacktriangle,\varpropto,C',\mu_3)$ is preserved in $V^5$, $\sfrak\leq\mu_3$ and $\lambda\leq\non(\Nwf),\rfrak$ by Lemma \ref{PreserTriangle}. Finally, because of the $\kappa$-cofinally many Hechler and Cohen reals added by the iteration, $\add(\Mwf)=\cof(\Mwf)=\kappa$.
\end{proof}
\begin{theorem}\label{dfrak,NonN,r-large}
   Assume $\cf(\lambda)\geq\mu_3$. It is consistent with $\thzfc$ that $\add(\Nwf)=\mu_1$, $\cov(\Nwf)=\mu_2$, $\pfrak=\sfrak=\bfrak=\mu_3$, $\cov(\Mwf)=\non(\Mwf)=\kappa$, $\dfrak=\non(\Nwf)=\rfrak=\cfrak=\lambda$ and $\mathbf{GCH}_\lambda$.
\end{theorem}
\begin{proof}
   Start with a model $V$ obtained by Theorem \ref{CovM-large}. Perform a fsi $\langle\Por_\alpha,\Qor_\alpha\rangle_{\alpha<\kappa}$ such that
   \begin{itemize}
      \item for $\alpha\equiv0\modulo5$, $\Qnm_\alpha=\Enm$ ($\Por_\alpha$-name for $\Eor$),
      \item for $\alpha\equiv1\modulo5$, $\Qnm_\alpha$ is the $\Por_\alpha$-name for the fsp of size $\lambda$
            of \emph{all} the subalgebras of $\Aor$ of size $<\mu_1$ in any $\Por_\alpha$-generic extension of $V$,
      \item for $\alpha\equiv2\modulo5$, $\Qnm_\alpha$ is the $\Por_\alpha$-name for fsp of size $\lambda$
            of \emph{all} the subalgebras of $\Bor$ of size $<\mu_2$ in any $\Por_\alpha$-generic extension of $V$,
      \item for $\alpha\equiv3\modulo5$, $\Qnm_\alpha$ is the $\Por_\alpha$-name for fsp of size $\lambda$
            of \emph{all} the subalgebras of $\Dor$ of size $<\mu_3$ in any $\Por_\alpha$-generic extension of $V$, and
      \item for $\alpha\equiv4\modulo5$, $\Qnm_\alpha$ is the $\Por_\alpha$-name for fsp of size $\lambda$
          of $\Mor_\Fwf$ for \emph{all} the filter bases $\Fwf$ of size $<\mu_3$ in any $\Por_\alpha$-generic extension of $V$.
   \end{itemize}
   By Lemmas \ref{smallPlus} and \ref{PresPlusUnbd->PresRhd}, $(+^{\mu_3}_{\Por_\kappa,\rhd})$ holds.
\end{proof}
\begin{theorem}\label{dfrak-large}
   Assume $\cf(\lambda)\geq\mu_2$. It is consistent with $\thzfc$ that $\add(\Nwf)=\mu_1$, $\pfrak=\bfrak=\sfrak=\mu_2$, $\cov(\Nwf)=\non(\Mwf)=\cov(\Mwf)=\non(\Nwf)=\kappa$, $\dfrak=\rfrak=\cfrak=\lambda$ and $\mathbf{GCH}_\lambda$.
\end{theorem}
\begin{proof}
   Start with a model $V$ of Theorem \ref{CovM-large} with $\mu_3=\mu_2$. Perform a fsi $\langle\Por_\alpha,\Qor_\alpha\rangle_{\alpha<\kappa}$ such that
   \begin{itemize}
      \item for $\alpha\equiv0\modulo4$, $\Qnm_\alpha=\Bnm$ ($\Por_\alpha$-name for $\Bor$),
      \item for $\alpha\equiv1\modulo4$, $\Qnm_\alpha$ is the $\Por_\alpha$-name for the fsp of size $\lambda$
            of \emph{all} the subalgebras of $\Aor$ of size $<\mu_1$ in any $\Por_\alpha$-generic extension of $V$,
      \item for $\alpha\equiv2\modulo4$, $\Qnm_\alpha$ is the $\Por_\alpha$-name for fsp of size $\lambda$
            of \emph{all} the subalgebras of $\Dor$ of size $<\mu_2$ in any $\Por_\alpha$-generic extension of $V$, and
      \item for $\alpha\equiv3\modulo4$, $\Qnm_\alpha$ is the $\Por_\alpha$-name for fsp of size $\lambda$
            of $\Mor_\Fwf$ for \emph{all} the filter bases $\Fwf$ of size $<\mu_2$ in any $\Por_\alpha$-generic extension of $V$.
   \end{itemize}
   Note that $(+^{\mu_2}_{\Por_\kappa,\rhd})$ holds.
\end{proof}
Now we turn into the consistency results that come from constructions of matrix iterations as explained in Context \ref{ContextMatrix}.  These correspond to extensions of the applications done in \cite[Sect. 6]{mejia} for the cardinal invariants in Cichon's diagram, but including, when possible, values for the cardinal invariants $\pfrak$, $\sfrak$, $\rfrak$ and $\ufrak$. Theorem \ref{MatrixCont-large} corresponds, respectively, to the results of \cite[Subsect. 6.1]{mejia}. In the same way, Theorem \ref{MatrixCofN-large} corresponds to \cite[Subsect. 6.2]{mejia} and Theorem \ref{MatrixNonN-large} corresponds to
\cite[Subsect. 6.3]{mejia}. More explicitly, the proof of each of the results that follows uses a matrix iteration that extend the one of its corresponding result in \cite[Sect. 6]{mejia} by including Mathias forcing and Laver forcing with a filter base (or with an ultrafilter) in its construction. Discussions about some values we did not get are included in Section \ref{SecQ}.\\
Define $t:\kappa\nu\to\kappa$ such that $t(\kappa\delta+\alpha)=\alpha$ for
$\delta<\nu$ and $\alpha<\kappa$. The product $\kappa\nu$, as all the products we are going to consider from now on, denotes ordinal
product. Also, fix a bijection $g:\lambda\to\kappa\times\lambda$ and
put $(\cdot)_0:\kappa\times\lambda\to\kappa$ the projection onto the first coordinate.
\begin{theorem}\label{MatrixCont-large}
   Assume that $\cf(\lambda)\geq\aleph_1$. It is consistent with $\thzfc$ that $\add(\Nwf)=\pfrak=\non(\Mwf)=\nu$, $\cfrak=\lambda$,
   $\mathbf{GCH}_\lambda$ and that one of the following statements hold.
   \begin{enumerate}[(a)]
      \item $\cov(\Mwf)=\cof(\Nwf)=\rfrak=\kappa$.
      \item $\cov(\Mwf)=\nu$ and $\dfrak=\rfrak=\non(\Nwf)=\cof(\Nwf)=\kappa$.
      \item $\non(\Nwf)=\ufrak=\nu$, $\dfrak=\cof(\Nwf)=\kappa$.
      \item $\non(\Nwf)=\nu$, $\dfrak=\rfrak=\cof(\Nwf)=\kappa$.
      \item $\cof(\Mwf)=\nu$ and $\non(\Nwf)=\rfrak=\cof(\Nwf)=\kappa$.
      \item $\cof(\Mwf)=\ufrak=\nu$ and $\non(\Nwf)=\cof(\Nwf)=\kappa$.
      \item $\cof(\Mwf)=\non(\Nwf)=\ufrak=\nu$, $\cof(\Nwf)=\kappa$.
   \end{enumerate}
\end{theorem}
We prove only one item of this Theorem, so the reader will know how to prove the other items by reference of its corresponding result in \cite{mejia} by extending the matrix iteration construction in a similar way. The same is done for the following two theorems of this section.
\begin{proof}
   We prove item (b). Start with $V$ a model from Theorem \ref{CovM-large} with $\mu_1=\mu_2=\mu_3=\aleph_1$. According to Context \ref{ContextMatrix}, construct a matrix iteration
   $\Por_{\kappa,\lambda\kappa\nu}=
   \langle\langle\Por_{\alpha,\xi},\Qnm_{\alpha,\xi}\rangle_{\xi<\lambda\kappa\nu}\rangle_{\alpha\leq\kappa}$ such that $S=\cj{\lambda\rho}{\rho<\kappa\nu}$, $U=L=\varnothing$ and the following for each $\rho<\kappa\nu$.
   \begin{enumerate}[(i)]
      \item $\Sor_{\lambda\rho}=\Eor$.
      \item If $\xi=\lambda\rho+1$, $\Anm_\rho$ is a $\Por_{t(\rho),\xi}$-name for $\Aor^{V_{t(\rho),\xi}}$ and
            \[\Qnm_{\alpha,\xi}=\left\{
               \begin{array}{ll}
                   \mathds{1} & \textrm{if $\alpha\leq t(\rho)$,}\\
                   \Anm_\rho & \textrm{if $\alpha>t(\rho)$.}
               \end{array}
               \right.\]
      \item If $\xi=\lambda\rho+2$, $\dot{\Uwf}_\rho$ is a $\Por_{t(\rho),\xi}$-name for an ultrafilter on $\omega$ and
            \[\Qnm_{\alpha,\xi}=\left\{
               \begin{array}{ll}
                   \mathds{1} & \textrm{if $\alpha\leq t(\rho)$,}\\
                   \Mor_{\Ufnm_\rho} & \textrm{if $\alpha>t(\rho)$.}
               \end{array}
               \right.\]
   \end{enumerate}
   For each $\alpha<\kappa$, fix a sequence $\langle\dot{\Fwf}^\rho_{\alpha,\gamma}\rangle_{\gamma<\lambda}$ of $\Por_{\alpha,\lambda\rho+3}$-names for \emph{all} the filter bases of size $<\nu$.
   \begin{enumerate}[(i)]
      \setcounter{enumi}{3}
      \item If $\xi=\lambda\rho+3+\epsilon$ ($\epsilon<\lambda$), put
      \[\Qnm_{\alpha,\xi}=\left\{
                            \begin{array}{ll}
                               \mathds{1}, & \textrm{if $\alpha\leq(g(\epsilon))_0$,}\\
                               \Mor_{\dot{\Fwf}^\rho_{g(\epsilon)}}, & \textrm{if $\alpha>(g(\epsilon))_0$.}
                            \end{array}
                     \right.\]
   \end{enumerate}
   By the same argument as in \cite[Thm. 12]{mejia}, all the statements, except $\nu\leq\pfrak$ and $\rfrak=\kappa$, hold in
   $V_{\kappa,\lambda\kappa\nu}$. For $\nu\leq\pfrak$, if $\Fwf$ is
   a filter base of size $<\nu$, by Theorem \ref{realsmatrixbelow} find $\alpha<\kappa$ and $\rho<\kappa\nu$ such that $\Fwf\in V_{\alpha,\lambda\rho}$. Then, there exists a $\gamma<\lambda$ such that $\Fwf=\Fwf^\rho_{\alpha,\gamma}$, so the Mathias real added by $\Mor_{\Fwf^\rho_{\alpha,\gamma}}$ is a pseudo-intersection of $\Fwf$.\\
   Lemma \ref{PresPlusUnbd->PresRhd} and Corollary \ref{MainMatrixConsq} gives $\kappa\leq\dfrak_\rhd=\min\{\dfrak,\rfrak\}$. For each $\rho<\kappa\nu$, let $m_\rho$ be the pseudo-intersection of $\Uwf_\rho$ added by $\Mor_{\Uwf_\rho}$.
   \begin{clm}\label{claimSplitting}
       Every family of $<\nu$ many infinite subsets of $\omega$ is $\varpropto$-bounded by some $m_\rho$
   \end{clm}
   \begin{proof}
       Let $\Cwf$ be such a family. By Theorem \ref{realsmatrixbelow}, we can find $\alpha<\kappa$ and $\eta<\kappa\nu$ such that
       $\Cwf\in V_{\alpha,\lambda\eta}$. Then, find $\rho\in(\eta,\kappa\nu)$ such that $t(\rho)=\alpha$. It is easy to conclude that $m_\rho$
       is a $\varpropto$-upper bound of $\Cwf$.
   \end{proof}
   This claim implies that $\rfrak\leq\kappa$.
\end{proof}
\begin{theorem}\label{MatrixCofN-large}
   Assume that $\cf(\lambda)\geq\mu_1$. It is consistent with $\thzfc$ that $\add(\Nwf)=\mu_1$, $\cov(\Nwf)=\pfrak=\non(\Mwf)=\nu$, $\cof(\Nwf)=\cfrak=\lambda$, $\mathbf{GCH}_\lambda$ and that one of the following statements hold.
   \begin{enumerate}[(a)]
      \item $\cov(\Mwf)=\cof(\Mwf)=\rfrak=\non(\Nwf)=\kappa$.
      \item $\cov(\Mwf)=\nu$ and $\non(\Nwf)=\dfrak=\rfrak=\cof(\Mwf)=\kappa$.
      \item $\cof(\Mwf)=\nu$ and $\non(\Nwf)=\rfrak=\kappa$.
      \item $\cof(\Mwf)=\ufrak=\nu$ and $\non(\Nwf)=\kappa$.
      \item $\non(\Nwf)=\ufrak=\nu$ and $\dfrak=\cof(\Mwf)=\kappa$.
      \item $\non(\Nwf)=\nu$ and $\dfrak=\rfrak=\cof(\Mwf)=\kappa$.
   \end{enumerate}
\end{theorem}
\begin{proof}
   We prove (d). Start with $V$ a model as in the conclusion of Theorem \ref{CovM-large} where $\mu_1=\mu_2=\mu_3$. Perform a matrix iteration
   $\Por_{\kappa,\lambda\kappa\nu}=
   \langle\langle\Por_{\alpha,\xi},\Qnm_{\alpha,\xi}\rangle_{\xi<\lambda\kappa\nu}\rangle_{\alpha\leq\kappa}$ where $S=U=\varnothing$, $L=\cj{\lambda\rho}{\rho<\kappa\nu}$, according to the following cases for $\rho<\kappa\nu$.
   \begin{enumerate}[(i)]
      \item If $\xi=\lambda\rho+1$, $\Bnm_\rho$ is a $\Por_{t(\rho),\xi}$-name for $\Bor^{V_{t(\rho),\xi}}$ and
            \[\Qnm_{\alpha,\xi}=\left\{
               \begin{array}{ll}
                   \mathds{1} & \textrm{if $\alpha\leq t(\rho)$,}\\
                   \Bnm_\rho & \textrm{if $\alpha>t(\rho)$.}
               \end{array}
               \right.\]
   \end{enumerate}
   For each $\alpha<\kappa$ fix a sequence $\langle\Anm_{\alpha,\gamma}^\rho\rangle_{\gamma<\lambda}$ of $\Por_{\alpha,\lambda\rho+2}$-names for \emph{all} the suborders of $\Aor^{V_{\alpha,\lambda\rho+2}}$ of size $<\mu_1$ and a sequence $\langle\dot{\Fwf}^\rho_{\alpha,\gamma}\rangle_{\gamma<\lambda}$ of $\Por_{\alpha,\lambda\rho+2}$-names for \emph{all} the filter bases of size $<\nu$.
   \begin{enumerate}[(i)]
      \setcounter{enumi}{1}
      \item If $\xi=\lambda\rho+2+2\epsilon$ ($\epsilon<\lambda$), put
            \[\Qnm_{\alpha,\xi}=\left\{
                            \begin{array}{ll}
                               \mathds{1}, & \textrm{if $\alpha\leq(g(\epsilon))_0$,}\\
                               \Anm^\rho_{g(\epsilon)}, & \textrm{if $\alpha>(g(\epsilon))_0$.}
                            \end{array}
                     \right.\]
      \item If $\xi=\lambda\rho+2+2\epsilon+1$ ($\epsilon<\lambda$), put
            \[\Qnm_{\alpha,\xi}=\left\{
                            \begin{array}{ll}
                               \mathds{1}, & \textrm{if $\alpha\leq(g(\epsilon))_0$,}\\
                               \Mor_{\dot{\Fwf}^\rho_{g(\epsilon)}}, & \textrm{if $\alpha>(g(\epsilon))_0$.}
                            \end{array}
                     \right.\]
   \end{enumerate}
   By similar arguments as in \cite[Thm. 18]{mejia} and Theorem \ref{MatrixCont-large}, we get that all the statements, except $\ufrak\leq\nu$, hold in $V_{\kappa,\lambda\kappa\nu}$. To see $\ufrak\leq\nu$, note that $\langle\ran(l_{\eta})\rangle_{\eta<\nu}$, where $l_\eta$ is the Laver real added by $\Lor_{\Uwf_{\kappa,\lambda\kappa\eta}}$, is a $\subseteq^*$-decreasing sequence that generates an ultrafilter.
\end{proof}
\begin{theorem}\label{MatrixNonN-large}
   Assume that $\cf(\lambda)\geq\mu_2$. It is consistent with $\thzfc$ that $\add(\Nwf)=\mu_1$, $\cov(\Nwf)=\mu_2$, $\pfrak=\non(\Mwf)=\nu$, $\non(\Nwf)=\cfrak=\lambda$, $\mathbf{GCH}_\lambda$ and that one of the following statements hold.
   \begin{enumerate}[(a)]
      \item $\cov(\Mwf)=\cof(\Mwf)=\rfrak=\kappa$.
      \item $\ufrak=\nu$ and $\dfrak=\cof(\Mwf)=\kappa$.
      \item $\cov(\Mwf)=\nu$ and $\dfrak=\rfrak=\cof(\Mwf)=\kappa$.
   \end{enumerate}
\end{theorem}
\begin{proof}
   To prove (b), assume that $V$ is a model as in the conclusion of Theorem \ref{CovM-large} with $\mu_2=\mu_3$. Perform a matrix iteration
   $\Por_{\kappa,\lambda\kappa\nu}=
   \langle\langle\Por_{\alpha,\xi},\Qnm_{\alpha,\xi}\rangle_{\xi<\lambda\kappa\nu}\rangle_{\alpha\leq\kappa}$ as explained in Context \ref{ContextMatrix} with $S=L=\varnothing$, $U=\cj{\lambda\rho}{\rho<\kappa\nu}$ and the following for each $\rho<\kappa\nu$.
   \begin{enumerate}[(i)]
      \item If $\xi=\lambda\rho+1$, $\Dnm_\rho$ is a $\Por_{t(\rho),\xi}$-name for $\Dor^{V_{t(\rho),\xi}}$ and
            \[\Qnm_{\alpha,\xi}=\left\{
               \begin{array}{ll}
                   \mathds{1} & \textrm{if $\alpha\leq t(\rho)$,}\\
                   \Dnm_\rho & \textrm{if $\alpha>t(\rho)$.}
               \end{array}
               \right.\]
   \end{enumerate}
   For each $\alpha<\kappa$ fix a sequence $\langle\Anm_{\alpha,\gamma}^\rho\rangle_{\gamma<\lambda}$ of $\Por_{\alpha,\lambda\rho+2}$-names for \emph{all} the suborders of $\Aor^{V_{\alpha,\lambda\rho+2}}$ of size $<\mu_1$, a sequence $\langle\Bnm_{\alpha,\gamma}^\rho\rangle_{\gamma<\lambda}$ of $\Por_{\alpha,\lambda\rho+2}$-names for \emph{all} the suborders of $\Bor^{V_{\alpha,\lambda\rho+2}}$ of size $<\mu_2$ and a sequence $\langle\dot{\Fwf}^\rho_{\alpha,\gamma}\rangle_{\gamma<\lambda}$ of $\Por_{\alpha,\lambda\rho+2}$-names for \emph{all} the filter bases of size $<\nu$.
   \begin{enumerate}[(i)]
      \setcounter{enumi}{1}
      \item If $\xi=\lambda\rho+2+3\epsilon$ ($\epsilon<\lambda$), put
            \[\Qnm_{\alpha,\xi}=\left\{
                            \begin{array}{ll}
                               \mathds{1}, & \textrm{if $\alpha\leq(g(\epsilon))_0$,}\\
                               \Anm^\rho_{g(\epsilon)}, & \textrm{if $\alpha>(g(\epsilon))_0$.}
                            \end{array}
                     \right.\]
      \item If $\xi=\lambda\rho+2+3\epsilon+1$ ($\epsilon<\lambda$), put
            \[\Qnm_{\alpha,\xi}=\left\{
                            \begin{array}{ll}
                               \mathds{1}, & \textrm{if $\alpha\leq(g(\epsilon))_0$,}\\
                               \Bnm^\rho_{g(\epsilon)}, & \textrm{if $\alpha>(g(\epsilon))_0$.}
                            \end{array}
                     \right.\]
      \item If $\xi=\lambda\rho+2+3\epsilon+2$ ($\epsilon<\lambda$), put
            \[\Qnm_{\alpha,\xi}=\left\{
                            \begin{array}{ll}
                               \mathds{1}, & \textrm{if $\alpha\leq(g(\epsilon))_0$,}\\
                               \Mor_{\dot{\Fwf}^\rho_{g(\epsilon)}}, & \textrm{if $\alpha>(g(\epsilon))_0$.}
                            \end{array}
                     \right.\]
   \end{enumerate}
   By the arguments from \cite[Thm. 21]{mejia} and Theorem \ref{MatrixCont-large}, $\add(\Nwf)=\mu_1$, $\cov(\Nwf)=\mu_2$,
   $\pfrak=\non(\Mwf)=\nu$, $\cov(\Mwf)\geq\nu$, $\dfrak=\kappa$, $\non(\Nwf)=\cfrak=\lambda$ and $\mathbf{GCH}_\lambda$ are true in $V_{\kappa,\lambda\kappa\nu}$. We are just left with the proof of $\ufrak\leq\nu$, but this is witnessed by $\langle{m'_\eta\rangle}_{\eta<\nu}$
   where each $m'_\eta$ is the Mathias real added by $\Mor_{\Uwf_{\kappa,\lambda\kappa\eta}}$. This sequence is $\subseteq^*$-decreasing and generates an ultrafilter on $\omega$.
\end{proof}

\section{Questions}\label{SecQ}

\begin{question}\label{QplusSuslin}
   Can we get $\rfrak=\kappa$ instead of $\ufrak=\nu$ in Theorem \ref{MatrixCont-large}(g)?
\end{question}
The main issue is that the matrix iteration construction associated to the proof of that statement involves the use of $\Bor$ and $\Dor$, as explained in Context \ref{ContextMatrix}(2)(i), in cofinally many columns of the matrix. But, as $(+_{\Bor,\varpropto})$ does \underline{not} hold, we cannot use Corollary \ref{MainMatrixConsq} to get $\rfrak\leq\kappa$ in any generic extension.\\
A natural question would be to ask for which relations $\sqsubset$ presented in Subsection \ref{SubsecPartCase} does Lemma \ref{mathiasStar} hold. It does \underline{not} hold for $\eqcirc$ because, if so, assuming $\nu<\kappa$, we could construct a model for which the statement of Theorem \ref{MatrixCont-large}(a) holds and $\ufrak\leq\nu$, which is a contradiction. Also, this Lemma does \underline{not} hold for $\varpropto$ because, if so, the model obtained by a matrix iteration of dimensions $\kappa\times\kappa\nu$ as in Context \ref{ContextMatrix} with $U=\kappa\nu$, $\nu<\kappa$, both regular uncountable, would yield a model of $\kappa\leq\rfrak$ and $\ufrak\leq\nu$, which is a contradiction to the fact that $\rfrak\leq\ufrak$. Likewise, it does \underline{not} hold for $\rhd$. Note that those arguments are valid to justify that Lemma \ref{LaverStarPitchfork}(c) does \underline{not} hold for $\eqcirc$, $\varpropto$ or $\rhd$.  So we are left with
\begin{question}\label{mathiasStarPitchfork}
   Does Lemma \ref{mathiasStar} hold for $\sqsubset=\pitchfork$.
\end{question}
A positive answer to this question will allow, with the presented method of matrix iterations, to obtain models of the statements of Theorem \ref{MatrixCont-large}(b) and Theorem \ref{MatrixCofN-large}(b) but with $\ufrak=\nu$ in place of $\rfrak=\kappa$. Note that this cannot be done by using Laver forcing as explained in Context \ref{ContextMatrix}(2)(iii) because this forcing notion adds dominating reals.\\
Note that, in the models where we get $\rfrak=\kappa$ for the items of Theorems \ref{MatrixCont-large}, \ref{MatrixCofN-large} and \ref{MatrixNonN-large}, we did not get a value for $\ufrak$.
\begin{question}\label{Qufrak}
   Can we get $\ufrak=\kappa$ or $\ufrak=\lambda$ in the corresponding consistency statements?
\end{question}

\bibliographystyle{spmpsci}      

\begin{thebibliography}{}
   \bibitem{baudor} Baumgartner, J., Dordal, P.: \emph{Adjoining dominating functions.} J. Symbolic Logic 50, no. 1, 94-101 (1985)
   \bibitem{bart} Bartoszynski, T.: \emph{Invariants of Measure and Category.} In: Kanamori, A., Foreman, M. (eds.) Handbook of Set-Theory,
                  pp. 491-555.  Springer, Heidelberg (2010)
   \bibitem{barju} Bartoszynski, T., Judah, H.: \emph{Set Theory. On the Structure of the Real Line.} A. K. Peters, Massachusetts (1995)
   \bibitem{blass} Blass, A.: \emph{Combinatorial cardinal characteristics of the continuum.} In: Kanamori, A., Foreman, M. (eds.) Handbook of Set-Theory,
                  pp. 395-490.  Springer, Heidelberg (2010)
   \bibitem{blass02} Blass, A.: \emph{Selective ultrafilters and homogeneity.} Ann. Pure Appl. Logic 38, no. 3, 215-255 (1988)
   \bibitem{blsh} Blass, A., Shelah, S.: \emph{Ultrafilters with small generating sets.} Israel J. Math.
                 65, 259-271 (1984)
   \bibitem{brendlebog} Brendle, J.: \emph{Forcing and the structure of the real line: the Bogot\'a lectures.} Lecture notes (2009)
   \bibitem{brendle} Brendle, J.: \emph{Larger cardinals in Cichon's diagram.} J. Symbolic Logic 56, no. 3, 795-810 (1991)
   \bibitem{brendle02} Brendle, J.: \emph{Mad families and ultrafilters.} Acta Universitatis Carolinae. Mathematica et Physica 49, 19-35 (2007)
   \bibitem{brendle03} Brendle, J.: \emph{Van Douwen's diagram for dense sets of rationals.} Ann. Pure Appl. Logic 143, 54-69 (2006)
   \bibitem{BF} Brendle, J., Fischer, V.: \emph{Mad families, splitting families
              and large continuum.} J. Symbolic Logic 76, no. 1, 198-208 (2011)
   \bibitem{gold} Goldstern, M.: \emph{Tools for your forcing construction.} In: Judah H. (ed.) Set theory of the reals,
                 pp. 305-360. Israel Mathematical Conference Proceedings, Bar Ilan University (1992)
   \bibitem{jech} Jech, T.: \emph{Set Theory.} 3rd millenium edition, Springer, Heidelberg (2002)
   \bibitem{jushe} Judah, H., Shelah, S.: \emph{The Kunen-Miller chart (Lebesgue measure, the Baire property, Laver reals and preservation
                  theorems for forcing).} J. Symbolic Logic 55, no. 3, 909-927 (1990)
   \bibitem{kamburelis} Kamburelis, A.: \emph{Iterations of boolean algebras with measure.} Arch. Math. Logic 29, 21-28 (1989)
   \bibitem{kamb} Kamburelis, A., Weglorz, B.: \emph{Splittings.} Arch. Math. Logic 35, 263-277 (1996)
   \bibitem{kunen} Kunen, K.: \emph{Set Theory: an introduction to independence proofs.} North-Holland, Amsterdam (1980)
   \bibitem{mallsh} Malliaris, M., Shelah, S.: \emph{Cofinality spectral theorems in model theory, set theory and general topology.} \texttt{http://arxiv.org/abs/1208.5424} (2012) Accessed 9 November 2012
   \bibitem{mejia} Mej\'ia, D.A.: \emph{Matrix iterations and Cichon's diagram.} Arch. Math. Logic 52, 261-278 (2013)
   \bibitem{miller} Miller, A.: \emph{Some properties of measure and category.} Trans. Amer. Math. Soc. 266, 93-114 (1981)
   \bibitem{shelah} Shelah, S.: \emph{Two cardinal invariants of the continuum ($\dfrak<\afrak$) and FS linearly ordered iterated forcing.} Acta Math. 192, 187-223 (2004) (publication number 700)
\end{thebibliography}


\end{document}